\documentclass[10pt]{article}
\usepackage{amsmath,amsfonts,latexsym, amssymb}
\usepackage{mathrsfs}
\usepackage{nicefrac}
\usepackage{graphicx}
\usepackage{amssymb, amsthm, tikz, xcolor}
\usetikzlibrary{arrows}
\usetikzlibrary{arrows.meta}

\def\spec{\mathrm{spec}}
\def\imp{\mathrm{imp}}

\newtheorem{theorem}{Theorem}
\newtheorem{definition}[theorem]{Definition}

\definecolor{light-gray}{gray}{0.3}
\definecolor{dark-gray}{gray}{0.6}
\definecolor{gray1}{gray}{0.2}
\definecolor{gray2}{gray}{0.4}
\definecolor{gray3}{gray}{0.6}
\definecolor{gray4}{gray}{0.8}

 \author{Jeffrey J. Beyerl
\thanks{Department of Mathematics, University of Central Arkansas, Conway, AR 72035
 \mbox{ email: \textsf{jbeyerl@uca.edu}}} }

\title{Stability of Critical $p$-Improper Interval Graphs}
\date{ }

\begin{document}

\maketitle
\begin{abstract}
A $p$-improper interval graph is an interval graph that has an interval representation in which no interval contains more than $p$ other intervals. A critical $p$-improper interval graph is $p-1$ improper when any vertex is removed. In this paper we investigate the spectrum of impropriety of critical $p$-improper interval graphs upon the removal of a single vertex, which is informally known as the stability of the graph.
\end{abstract}

\section{Introduction}
Interval graphs are a well studied class of graphs, having been classified and investigated thoroughly in the latter half of the 1900s. Any hereditary class of graphs necessarily has a minimal forbidden subgraph characterization, in the case of interval graphs that characterization is cordless cycles and asteroidal triples. Most introductory texts on graph theory such as \cite{West} will define the basic terminology needed. There are a number of other equivalent characterizations throughout the literature, such as \cite{IntervalGraphCharacterization1, IntervalGraphCharacterization2} and \cite{IntervalGraphCharacterization3}.

Proper interval graphs were introduced and classified in \cite{Roberts}. The classification is quite simple: in addition to being an interval graph, merely forbid the claw $K_{1,3}$. Both $q$-proper \cite{Proskurowski} and $p$-improper {\cite{Beyerl} interval graphs generalize this notion with 0-proper and 0-improper interval graphs being proper interval graphs.

In the same papers as above the classification for proper and improper interval graphs are investigated. The class of $q$-proper interval graphs has a rather straightforward classification: the additional minimal forbidden subgraph being a claw $K_{1,3}$ with one of the leaves replaced with the clique $K_{q+1}$. The classification for improper interval graphs appears to be somewhat complicated and to date has not been written down, although some partial results have been obtained in \cite{Beyerl} and \cite{Beyerl2}.

Because of the hereditary nature of improper interval graphs, every $(p-1)$-improper interval graph is also $p$-improper. Oftentimes we will want to specify a graph that is $p$-improper but not $(p-1)$-improper, and so we call such a graph \textit{exactly} $p$-improper.
 
The \textit{basepoint} of an improper interval graph, as well as \textit{exterior local components} and \textit{local components} are formally defined in \cite{Beyerl} in terms of some technical concepts which are not needed in this paper. We provide an informal illustration of these terms in Figure \ref{defns}.

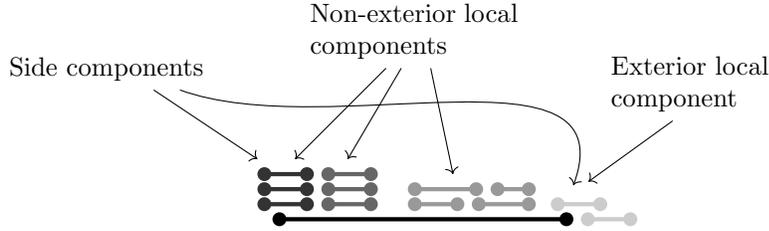
\begin{figure}
\begin{center}
\begin{tikzpicture}
[IntervalGraph/.style={{Circle[scale=0.7]}-{Circle[scale=0.7]}, ultra thick}]
	\draw [IntervalGraph] (-2,0) -- (2,0);
	%K_n left
	\draw [gray2, IntervalGraph] (-1.35,0.2) -- (-0.6,.2);
	\draw [gray2, IntervalGraph] (-1.35,0.4) -- (-0.6,.4);
	\draw [gray2, IntervalGraph] (-1.35,0.6) -- (-0.6,.6);
	%K_n right
	\draw [gray4, IntervalGraph] (1.7,0.2) -- (2.45,.2);	
	\draw [gray4, IntervalGraph] (2.1,0.0) -- (2.85,.0);	
	\draw [gray3, IntervalGraph] (0.65,0.2) -- (1.5,.2); %right
	\draw [gray3, IntervalGraph] (-.2,0.2) -- (.55,.2);  %left
	\draw [gray3, IntervalGraph] (-0.2,0.4) -- (0.8,0.4);
	\draw [gray3, IntervalGraph] (0.9,0.4) -- (1.5,0.4);	
	%K_n leftmost
	\draw [gray1, IntervalGraph] (-1.45,0.2) -- (-2.2,.2);
	\draw [gray1, IntervalGraph] (-1.45,0.4) -- (-2.2,.4);
	\draw [gray1, IntervalGraph] (-1.45,0.6) -- (-2.2,.6);
	
%	\draw[text width=3cm] (4, 1) node[above] {Exterior local \\ component};
	
	\node[text width=3cm] (EL) at (4, 1.8) {Exterior local \\ component};

	\node[text width=3cm] (NEL) at (0, 2.5) {Non-exterior local \\ components};

	\node[text width=3cm] (SC) at (-4, 2) {Side components};
	\draw[->] (SC) -- (-2.2,0.8);
	\draw[->] (NEL) -- (-1.7,0.8);
	\draw[->] (NEL) -- (-1,0.8);
	\draw[->] (NEL) -- (0.4,0.6);
	\draw[->] (EL) -- (2.2,0.5);
	\draw[->] (SC) to [out=-20, in=70] (2,0.45);
\end{tikzpicture}
\caption{The basepoint is in black. If removed, the remaining components are called local components. \textit{Basepoint}, \textit{exterior} and \textit{non-exterior} depend only on the graph. However, which local components are \textit{side components} depends on the representation.}
\label{defns}
\end{center}
\end{figure}

%\\ define basepoint, exterior local component, local component
\section{Instability of graphs with a fixed \\ impropriety}
The primary focus of this paper is the notion of \textit{stability} of improper interval graphs. By this we are asking what impropriety can be obtained by removing a vertex from an improper interval graph. For comparison, consider $q$-proper interval graphs, which are very stable: when a vertex from a critical exactly $q$-proper interval graph is removed, the result is either $0$-proper or $(q-1)$-proper. This can easily be seen by considering the minimal forbidden subgraph for $q$-proper interval graphs. For improper interval graphs, it is somewhat more complicated. We formalize the notion of stability in terms of the spectrum of impropriety below. 

\begin{definition}
Let $G$ be a critical $p$-improper interval graph. Then the spectrum of impropriety is the set of improprieties that can be exactly obtained from $G$ by the removal of a single vertex and is denoted $\spec_{\imp}(G)$. The spectrum of impropriety for the class of $p$-improper interval graphs is the set of improprieties that can be exactly obtained by the removal of a single vertex from some critical exactly $p$-improper interval graph and is denoted $\spec_{\imp}(p)$. 

Note in particular that the subscript $\imp$ is meant to disassociate this from the unrelated concept of the spectrum of a graph that consists of eigenvalues.
\end{definition}

The theorem below illustrates the instability of improper interval graphs by showing that given a critical $p$-improper interval graph, if a single vertex is removed there are no meaningful bounds on the impropriety of the resulting graph.

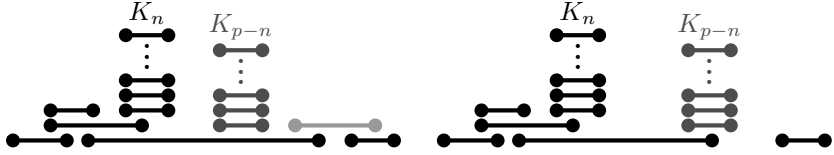
\begin{figure}
\begin{center}
\begin{tabular}{cc}
\begin{tikzpicture}
[IntervalGraph/.style={{Circle[scale=0.7]}-{Circle[scale=0.7]}, ultra thick}]
\draw [IntervalGraph] (-1.25,0) -- (2,0);
%left
\draw [IntervalGraph] (-2.25,0) -- (-1.35,0);
\draw [IntervalGraph] (-1.75,0.2) -- (-0.35,0.2);
\draw [IntervalGraph] (-1.75,0.4) -- (-1,0.4);
%K_n
\draw [IntervalGraph] (-.75,0.4) -- (0,.4);
\draw [IntervalGraph] (-.75,0.6) -- (0,.6);
\draw [IntervalGraph] (-.75,0.8) -- node[above] {\textbf{$\vdots$}} (0,.8);
\draw [IntervalGraph] (-.75,1.4) -- node[above]{$K_n$} (0,1.4);
%K_{p-n}
\draw [light-gray, IntervalGraph] (0.5,0.2) -- (1.25,.2);
\draw [light-gray, IntervalGraph] (0.5,0.4) -- (1.25,.4);
\draw [light-gray, IntervalGraph] (0.5,0.6) --node[above] {\textbf{$\vdots$}}  (1.25,.6);
\draw [light-gray, IntervalGraph] (0.5,1.2) -- node[above]{$K_{p-n}$} (1.25,1.2);
%right
\draw [IntervalGraph] (2.25,0) -- (3,0);
\draw [dark-gray, IntervalGraph] (2.75,0.2) -- (1.5,0.2);
\end{tikzpicture}
&
\begin{tikzpicture}
[IntervalGraph/.style={{Circle[scale=0.7]}-{Circle[scale=0.7]}, ultra thick}]
\draw [IntervalGraph] (-1.25,0) -- (1.5,0);
%left
\draw [IntervalGraph] (-2.25,0) -- (-1.35,0);
\draw [IntervalGraph] (-1.75,0.2) -- (-0.35,0.2);
\draw [IntervalGraph] (-1.75,0.4) -- (-1,0.4);
%K_n
\draw [IntervalGraph] (-.75,0.4) -- (0,.4);
\draw [IntervalGraph] (-.75,0.6) -- (0,.6);
\draw [IntervalGraph] (-.75,0.8) -- node[above] {\textbf{$\vdots$}} (0,.8);
\draw [IntervalGraph] (-.75,1.4) -- node[above]{$K_n$} (0,1.4);
%K_{p-n}
\draw [light-gray, IntervalGraph] (1,0.2) -- (1.75,.2);
\draw [light-gray, IntervalGraph] (1,0.4) -- (1.75,.4);
\draw [light-gray, IntervalGraph] (1,0.6) --node[above] {\textbf{$\vdots$}}  (1.75,.6);
\draw [light-gray, IntervalGraph] (1,1.2) -- node[above]{$K_{p-n}$} (1.75,1.2);
%right
\draw [IntervalGraph] (2.25,0) -- (3,0);
\end{tikzpicture}
\end{tabular}

\caption{In this graph, when the light grey interval is removed, the clique $K_{p-n}$ becomes a potential side component, leaving the resulting graph exactly $n$-improper.}
\label{dropToN}
\end{center}
\end{figure}

\begin{theorem}\label{mainThm}
$\spec_{\imp}(p)=\{0, 1, 2, ..., p-1\}$
\end{theorem}

\begin{proof}
Given a number $0\leq n\leq p-1$, we will construct a critical $p$-improper interval graph from which we can remove a vertex to obtain an exactly $n$-improper interval graph. Such a graph is illustrated in Figure \ref{dropToN} and consists of a basepoint with two exterior local components: one a $P_2$, and the other a $K_3$ that is connected to a clique $K_{n}$ with one edge only. One of the other vertices in the $K_3$ is not adjacent to the basepoint. Without loss of generality, assume that the $K_3$ is to the left of the basepoint, so that we may represent its vertices as the intervals $(a,b_1), (a,b_2), (a,b_3)$ where $b_1<B<b_2<b_3$; $B$ is the left endpoint of the basepoint. Then if $C$ is the leftmost endpoint of intervals representing the clique $K_n$, we must have $b_2<C<b_3$ so that the vertices in the clique are adjacent to one but not more of the vertices in the $K_3$. Hence, the entire clique will be contained in the basepoint because $B<C$ and there is more than one local component. This is illustrated in Figure \ref{dropToN}. The third local component is a clique $K_{p-n}$. When the light grey vertex is removed, the graph has the representation given on the right of the figure that is now exactly $n$-improper.
\end{proof}

We also remark that the graph created in the above proof is not critical because of the remaining vertex not connected to the basepoint. A similar construction for $n\geq 1$ using cliques of size $K_{n-1}$ and $K_{p-n+1}$ will construct a critical $n$-improper interval graph when the other vertex in the $P_2$ side component is removed.

Note that the construction of a graph with full spectrum utilized exterior local components. Exterior local components often make interval graphs easier to analyze because they are simpler. If we do not allow exterior local components, we can still construct graphs that give the desired spectrum. For example, the class of graphs in Figure \ref{dropHalf} show that the spectrum includes $\{0, 1, 2, ..., \lfloor\frac{p}{2}\rfloor \}$; in Figure \ref{dropAll} we see that it has full spectrum. These two claims can be proven using similar methods as Theorem \ref{mainThm}.

Turning the problem on its head, it is sometimes possible to construct graphs with a given impropriety. For instance Figure \ref{DropToSeven} provides, for every $p\geq 8$, a class of critical $p$-improper interval graphs that have impropriety 7 after the removal of one vertex.

\begin{figure}
\begin{center}
\begin{tabular}{ccccc}
\begin{tikzpicture}
[IntervalGraph/.style={{Circle[scale=0.7]}-{Circle[scale=0.7]}, ultra thick}]
	\draw [IntervalGraph] (-1.15,0) -- (2,0);
	%left
	\draw [IntervalGraph] (-1.25,0.2) -- (1,0.2);
	%K_n 1	
	\draw [IntervalGraph] (-1.35,0.4) -- (-0.6,.4);
	\draw [IntervalGraph] (-1.35,0.6) -- (-0.6,.6);
	\draw [IntervalGraph] (-1.35,0.8) -- (-0.6,.8);
	\draw [IntervalGraph] (-1.35,1) -- node[above] {\textbf{$\vdots$}} (-0.6,1);
	\draw [IntervalGraph] (-1.35,1.6) -- node[above]{$K_{p-n}$} (-0.6,1.6);

	%K_n 2	
	\draw [light-gray, IntervalGraph] (-0.5,0.4) -- (0.25,.4);
	\draw [light-gray, IntervalGraph] (-.5,0.6) -- (0.25,.6);
	\draw [light-gray, IntervalGraph] (-.5,0.8) -- (0.25,.8);
	\draw [light-gray, IntervalGraph] (-.5,1) -- node[above] {\textbf{$\vdots$}} (0.25,1);
	\draw [light-gray, IntervalGraph] (-.5,1.6) -- node[above]{$K_{p-n}$} (0.25,1.6);

	%K_n 3	
	\draw [IntervalGraph] (.35,0.4) -- (1.1,.4);
	\draw [IntervalGraph] (.35,0.6) -- (1.1,.6);
	\draw [IntervalGraph] (.35,0.8) -- node[above] {\textbf{$\vdots$}} (1.1,.8);
	\draw [IntervalGraph] (.35,1.4) -- node[above]{$K_n$} (1.1,1.4);

	%right
	\draw [dark-gray, IntervalGraph] (1.5,0.2) -- (2.25,0.2);
\end{tikzpicture}
& & & &
\begin{tikzpicture}
[IntervalGraph/.style={{Circle[scale=0.7]}-{Circle[scale=0.7]}, ultra thick}]
	\draw [IntervalGraph] (-1.15,0) -- (2,0);
	%left
	\draw [IntervalGraph] (-1.25,0.2) -- (1,0.2);
	%K_n 1	
	\draw [IntervalGraph] (-1.35,0.4) -- (-0.6,.4);
	\draw [IntervalGraph] (-1.35,0.6) -- (-0.6,.6);
	\draw [IntervalGraph] (-1.35,0.8) -- (-0.6,.8);
	\draw [IntervalGraph] (-1.35,1) -- node[above] {\textbf{$\vdots$}} (-0.6,1);
	\draw [IntervalGraph] (-1.35,1.6) -- node[above]{$K_{p-n}$} (-0.6,1.6);

	%K_n 2	
	\draw [IntervalGraph] (-0.5,0.4) -- (0.25,.4);
	\draw [IntervalGraph] (-.5,0.6) -- (0.25,.6);
	\draw [IntervalGraph] (-.5,0.8) -- node[above] {\textbf{$\vdots$}} (0.25,0.8);
	\draw [IntervalGraph] (-.5,1.4) -- node[above]{$K_n$} (0.25,1.4);

	%K_n 3	
	\draw [light-gray, IntervalGraph] (.35,0.4) -- (2.25,.4);
	\draw [light-gray, IntervalGraph] (.35,0.6) -- (2.25,.6);
	\draw [light-gray, IntervalGraph] (.35,0.8) -- (2.25,.8);
	\draw [light-gray, IntervalGraph] (.35,1) -- node[above] {\textbf{$\vdots$}} (2.25,1);
	\draw [light-gray, IntervalGraph] (.35,1.6) -- node[above]{$K_{p-n}$} (2.25,1.6);

\end{tikzpicture}
\end{tabular}

\caption{In this graph, when the light grey interval is removed, the clique $K_{p-n}$ may be swapped to the right side and extended beyond the basepoint, leaving the resulting graph exactly $n$-improper, for $n\leq \lfloor \frac{p}{2} \rfloor$. The resulting graph is not critical, however, because it has two base points - one of which could be removed and it would still be exactly $n$-improper.}
\label{dropHalf}
\end{center}
\end{figure}
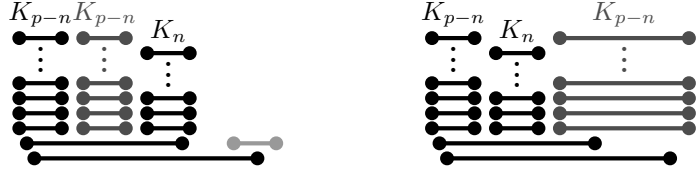

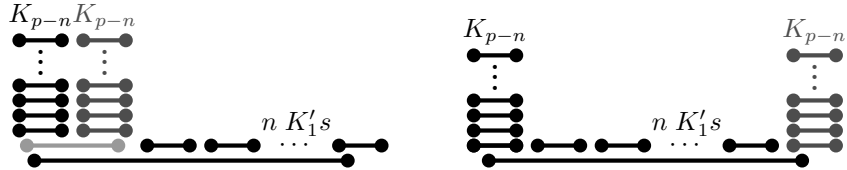
\begin{figure}
\begin{center}
\begin{tabular}{ccc}
\begin{tikzpicture}
[IntervalGraph/.style={{Circle[scale=0.7]}-{Circle[scale=0.7]}, ultra thick}]
	\draw [IntervalGraph] (-1.15,0) -- (3.2,0);
	%left
	\draw [dark-gray, IntervalGraph] (-1.25,0.2) -- (0.15,0.2);
	%K_n 1	
	\draw [IntervalGraph] (-1.35,0.4) -- (-0.6,.4);
	\draw [IntervalGraph] (-1.35,0.6) -- (-0.6,.6);
	\draw [IntervalGraph] (-1.35,0.8) -- (-0.6,.8);
	\draw [IntervalGraph] (-1.35,1) -- node[above] {\textbf{$\vdots$}} (-0.6,1);
	\draw [IntervalGraph] (-1.35,1.6) -- node[above]{$K_{p-n}$} (-0.6,1.6);

	%K_n 2	
	\draw [light-gray, IntervalGraph] (-0.5,0.4) -- (0.25,.4);
	\draw [light-gray, IntervalGraph] (-.5,0.6) -- (0.25,.6);
	\draw [light-gray, IntervalGraph] (-.5,0.8) -- (0.25,.8);
	\draw [light-gray, IntervalGraph] (-.5,1) -- node[above] {\textbf{$\vdots$}} (0.25,1);
	\draw [light-gray, IntervalGraph] (-.5,1.6) -- node[above]{$K_{p-n}$} (0.25,1.6);

	\draw [IntervalGraph] (0.35,0.2) -- (1.1,.2);
	\draw [IntervalGraph] (1.2,0.2) -- (1.95,.2);
	\draw [white] (2.05,0.2) -- node[black] {$\cdots$} node[above, black] {$n \;K_1's$} (2.8,.2);
	%right
	\draw [IntervalGraph] (2.9,0.2) -- (3.65,0.2);
\end{tikzpicture}
& &
\begin{tikzpicture}
[IntervalGraph/.style={{Circle[scale=0.7]}-{Circle[scale=0.7]}, ultra thick}]
	\draw [IntervalGraph] (-1.15,0) -- (3.2,0);
	%K_n 1	
	\draw [IntervalGraph] (-1.35,0.2) -- (-0.6,.2);
	\draw [IntervalGraph] (-1.35,0.4) -- (-0.6,.4);
	\draw [IntervalGraph] (-1.35,0.6) -- (-0.6,.6);
	\draw [IntervalGraph] (-1.35,0.8) -- node[above] {\textbf{$\vdots$}} (-0.6,.8);
	\draw [IntervalGraph] (-1.35,1.4) -- node[above]{$K_{p-n}$} (-0.6,1.4);

	%K_n 2	
	\draw [light-gray, IntervalGraph] (2.9,0.2) -- (3.65,.2);
	\draw [light-gray, IntervalGraph] (2.9,0.4) -- (3.65,.4);
	\draw [light-gray, IntervalGraph] (2.9,0.6) -- (3.65,.6);
	\draw [light-gray, IntervalGraph] (2.9,.8) -- node[above] {\textbf{$\vdots$}} (3.65,.8);
	\draw [light-gray, IntervalGraph] (2.9,1.4) -- node[above]{$K_{p-n}$} (3.65,1.4);

	\draw [IntervalGraph] (0.35,0.2) -- (1.1,.2);
	\draw [IntervalGraph] (-.5,0.2) -- (0.25,.2);
	\draw [white] (1.2,0.2) -- node[black] {$\cdots$} node[above, black] {$n \;K_1's$} (1.95,.2);
	\draw [IntervalGraph] (2.05,0.2) -- (2.8,.2);
	%right
	\draw [IntervalGraph] (-1.35,0.2) -- (-0.6,.2);
\end{tikzpicture}
\end{tabular}
\caption{In this graph, when the light grey interval is removed, the clique $K_{p-n}$ may be swapped to the right side, leaving the resulting graph exactly $n$-improper.}
\label{dropAll}
\end{center}
\end{figure}

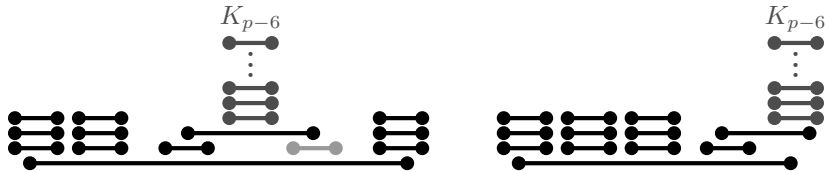
\begin{figure}
\begin{center}
\begin{tabular}{ccc}
\begin{tikzpicture}
[IntervalGraph/.style={{Circle[scale=0.7]}-{Circle[scale=0.7]}, ultra thick}]
	\draw [IntervalGraph] (-2,0) -- (3.2,0);
	%K_n left
	\draw [IntervalGraph] (-1.35,0.2) -- (-0.6,.2);
	\draw [IntervalGraph] (-1.35,0.4) -- (-0.6,.4);
	\draw [IntervalGraph] (-1.35,0.6) -- (-0.6,.6);
	%K_n right
	\draw [IntervalGraph] (2.65,0.2) -- (3.4,.2);	
	\draw [IntervalGraph] (2.65,0.4) -- (3.4,.4);	
	\draw [IntervalGraph] (2.65,0.6) -- (3.4,.6);	
	
	\draw [IntervalGraph] (.1,0.4) -- (1.95,.4); %big		
	\draw [dark-gray, IntervalGraph] (1.5,0.2) -- (2.25,.2); %right
	\draw [IntervalGraph] (-.2,0.2) -- (.55,.2);  %left
	\draw [light-gray, IntervalGraph] (.65,0.6) -- (1.4,.6);
	\draw [light-gray, IntervalGraph] (.65,0.8) -- (1.4,.8);
	\draw [light-gray, IntervalGraph] (.65,1) --  node[above] {\textbf{$\vdots$}}(1.4,1);
	\draw [light-gray, IntervalGraph] (.65,1.6) -- node[above]{$K_{p-6}$}(1.4,1.6);
	%K_n leftmost
	\draw [IntervalGraph] (-1.45,0.2) -- (-2.2,.2);
	\draw [IntervalGraph] (-1.45,0.4) -- (-2.2,.4);
	\draw [IntervalGraph] (-1.45,0.6) -- (-2.2,.6);
\end{tikzpicture}
& &
\begin{tikzpicture}
[IntervalGraph/.style={{Circle[scale=0.7]}-{Circle[scale=0.7]}, ultra thick}]
	\draw [IntervalGraph] (-2,0) -- (1.8,0);
	%K_n left
	\draw [IntervalGraph] (-1.35,0.2) -- (-0.6,.2);
	\draw [IntervalGraph] (-1.35,0.4) -- (-0.6,.4);
	\draw [IntervalGraph] (-1.35,0.6) -- (-0.6,.6);
	%K_n left2
	\draw [IntervalGraph] (-.5,0.2) -- (.25,.2);	
	\draw [IntervalGraph] (-.5,0.4) -- (.25,.4);	
	\draw [IntervalGraph] (-.5,0.6) -- (.25,.6);	
	%Main
	\draw [IntervalGraph] (0.7,0.4) -- (2.05,.4); %big		
	\draw [IntervalGraph] (0.5,0.2) -- (1.25,.2);  %left
	\draw [light-gray, IntervalGraph] (1.4,0.6) -- (2.15,.6);
	\draw [light-gray, IntervalGraph] (1.4,0.8) -- (2.15,.8);
	\draw [light-gray, IntervalGraph] (1.4,1) --  node[above] {\textbf{$\vdots$}}(2.15,1);
	\draw [light-gray, IntervalGraph] (1.4,1.6) -- node[above]{$K_{p-6}$}(2.15,1.6);
	%K_n leftmost
	\draw [IntervalGraph] (-1.45,0.2) -- (-2.2,.2);
	\draw [IntervalGraph] (-1.45,0.4) -- (-2.2,.4);
	\draw [IntervalGraph] (-1.45,0.6) -- (-2.2,.6);
\end{tikzpicture}
\end{tabular}
\caption{In this graph, when the light grey interval is removed, the clique $K_{p-6}$ may be swapped to the right side and extended beyond the basepoint, leaving the resulting graph exactly $7$-improper.}
\label{DropToSeven}
\end{center}
\end{figure}

\section{Instability of a fixed graph}

At the start of the previous section we defined the spectrum of impropriety for the class of $p$-improper interval graphs, $\spec_{\imp}(p)$ as well as the spectrum of impropriety of an individual graph $G$, $\spec_{\imp}(G)$. In this section we look at $\spec_{\imp}(G)$. If a balanced interval graph has two exterior local components, then $\spec_{\imp}(G)$ contains only a few numbers, given by the theorem below. If not, it is not entirely clear. We conjecture that $\left |\spec_{\imp}(G)\right |$ is unbounded but rather small in comparison to $p$; exactly how small (or large) is currently an open problem. 

\begin{theorem}
Suppose $G$ is a balanced improper interval graph with two exterior local components. Then $\left |\spec_{\imp}(G)\right |\leq 4$.
\end{theorem}

\begin{proof}
Suppose $G$ is a balanced improper interval graph with two exterior local components. The fact that there are two exterior local components means that all of the other local components are contained within the basepoint in every interval representation. If a vertex on one of these non-exterior local components is removed, then the impropriety decreases by exactly one, as there are still two exterior local components keeping the other local components contained within the basepoint in every interval representation.

Theorem 3.1 in \cite{Beyerl} tells us that the exterior local components themselves are $P_2$. The two vertices in the $P_2$ each potentially give us another value for $\spec_{\imp}(G)$. By the symmetry of the two exterior local components, that is all we get from the $P_2$'s. The fourth value is from removing the only vertex we have not yet considered, the basepoint itself. 
\end{proof}

\bibliographystyle{plain}
\bibliography{References}

\end{document}